\documentclass[11pt]{amsart}
\usepackage{xcolor}
\usepackage{graphicx}
\usepackage{dsfont}
\usepackage{amsmath, amssymb}
\usepackage{float}
\usepackage{bbm}
\usepackage{todonotes}
\usepackage{enumerate}

\newtheorem{theorem}{Theorem}

\newtheorem{prop}{Proposition}

\newtheorem{remark}{Remark}

\newcommand{\be}{\begin{equation}}
\newcommand{\ee}{\end{equation}}
\newcommand{\bea}{\begin{eqnarray}}
\newcommand{\eea}{\end{eqnarray}}
\newcommand{\barr}{\begin{array}}
\newcommand{\earr}{\end{array}}

\newcommand{\bpar}{\begin{equation} \left\{ \begin{array}{lll}}
\newcommand{\epar}{ \end{array}\right. \end{equation} }
\newcommand{\eparn}{ \end{array} \right.}

\newcommand\leftmat{\left(\begin{array}{cc}}
\newcommand\rightmat{\end{array}\right)}
\newcommand\leftvec{\left(\begin{array}{c}}
\newcommand\rightvec{\end{array}\right)}
\newcommand\re{{\rm e}}

\newcommand\R{{\mathbb R}}
\newcommand\C{{\mathbb C}}

\newcommand\N{{\mathbb N}}

\def\XXint#1#2#3{{\setbox0=\hbox{$#1{#2#3}{\int}$}
     \vcenter{\hbox{$#2#3$}}\kern-.5\wd0}}

\providecommand{\bigoh[1]}{\bigohsymbolonly\left(#1\right)}

\providecommand{\abs[1]}{\left\lvert#1\right\rvert}

\providecommand{\D}{\mathrm{d}}

\providecommand{\Lebesgue}{\operatorname{L}} 
\providecommand{\ContinuousSpace}{\operatorname{C}} 
\providecommand{\BV}{\operatorname{BV}}       

\providecommand{\BVzo}{\BV[0,1]}       
\providecommand{\Ltwozo}{\Lebesgue^2(0,1)}       
\providecommand{\Ctwozo}{\ContinuousSpace^2[0,1]}
\providecommand{\Czopen}{\ContinuousSpace(0,1)}

\providecommand{\argdot}{{}\cdot{}}

\renewcommand{\Re}{\operatorname{Re}}
\renewcommand{\Im}{\operatorname{Im}}

\title[Revivals for the Airy equation]{Revivals, or the Talbot effect, for the Airy equation}
\author[]{B. Pelloni$^\pi$ \& D.A. Smith$^\sigma$
}
\address{
$^{\pi}$Heriot-Watt University \& Maxwell Institute  for the Mathematical Sciences, Edinburgh, Scotland. $^\sigma$Yale-NUS College \& National University of Singapore, Singapore.
}

\email{$^\pi$B.Pelloni@hw.ac.uk, $^\sigma$dave.smith@yale-nus.edu.sg}

\date{\today}

\begin{document}
\maketitle

\begin{abstract}
We study Dirichlet-type problems for the simplest third-order linear dispersive PDE, often referred to as the Airy equation.
Such problems have not been extensively studied, perhaps due to the complexity of the spectral structure of the spatial operator. Our specific interest is to determine whether the peculiar phenomenon of revivals, also known as Talbot effect, is supported by these boundary conditions, which for third order problems are not reducible to periodic ones.   We prove that this is the case only for a very special choice of the boundary conditions, for which a new type of weak cusp revival phenomenon has been recently discovered \cite{BFPS}. We also give some new results on the functional class of the solution for other cases.

\medskip
{\em In memory of Vassilis Dougalis, a wonderful teacher and dear friend}
\end{abstract}
\section{Introduction}
The phenomenon of {\em revivals}
in linear dispersive periodic problems,
originally named {\em Talbot effect} and also known as {\em dispersive quantisation},
 has been thoroughly studied in the periodic setting, and is by now well-understood, both from a theoretical and from an experimental point of view. An accepted working definition of this phenomenon states that at specific values of the time variable, called {\em rational times}, and only at these times, the solution is  a linear superposition of a finite number of translated copies of the initial datum.
 This behaviour is particularly striking when the initial profile has jump discontinuities, as these jumps then propagate in the solution at rational times, while at  any other time, jump discontinuities in the initial datum generate a fractal solution profile. This phenomenon is known as {\em fractalisation}. For a recent review and bibliography, see \cite{smith2020revival}.

 The focus of this paper is the case when the boundary conditions are not periodic.
The particular case we study is the Airy equation posed on a finite interval, which we always specify as the interval $(0,1)$:
\begin{align}
&u_t+u_{xxx}=0
\qquad &&x\in(0,1), \;0<t<\infty, \nonumber
\\
&u(\cdot,0)=u_0,\quad u_0\in \Ltwozo, && x\in(0,1).
\label{pseD0}\\
&3\mbox{ boundary conditions at }x=0\mbox{ and }x=1,
\quad\;&& t>0.\nonumber
\end{align}

The case of periodic boundary conditions was first studied in \cite{olver2010dispersive}, and can also be seen as a particular  example of the general theory given in \cite{erdougan2016dispersive}. In this case, both the revival property and its dual phenomenon of fractalisation hold. However, perturbing even slightly the periodicity of the problem can destroy the support for any form of  revivals. For example, if the boundary conditions are quasi-periodic, namely of  the form
\be\label{qp}
\partial_x^ku(0,t)=\re^{i\theta \pi}\partial_x^ku(1,t), \qquad k=0,1,2,\quad t>0,\quad \theta \in
\R,
\ee
the revival property holds only if $\theta\in\mathbb Q$, see \cite{BFP,farmakis2024}.  This is in marked contrast with the case of second order linear dispersive PDEs, when the quasi-periodic boundary problem can be reduced to a periodic one and therefore always supports the revival property.

Another natural class of boundary conditions, at least for second order problems, is defined by the requirement that  the solution vanish at the endpoints. For second order problems, these conditions are called homogeneous Dirichlet boundary conditions, and the resulting boundary value problem can be formulated as a problem periodic on a larger interval, hence it always supports revivals.

This paper is concerned with the ``Dirichlet'' problem for the Airy equation. Since the equation is of third order, it is necessary to prescribe three boundary conditions,  therefore the usual Dirichlet boundary  conditions $u(0,t)=u(1,t)=0$ must be complemented with a third condition. We will call the resulting problem of {\em Dirichlet-type}, and the question we address is whether such a problem supports revivals in the third-order case, in particular for the Airy equation.
Here we study specifically the problem determined by the boundary conditions 
\be\label{dirbc}
u(0,t)=u(1,t)=0,\quad u_x(1,t)=\beta u_x(0,t),\qquad   t>0,\quad \beta\in\R.
\ee
As we argue below,  without loss of generality we can restrict attention to the case $\beta\in[0,1]$. We then show the following:
\begin{itemize}
\item
For $\beta=0$, we indicate that the problem does not support revivals, at least in a weak sense, by showing that the $L^2(0,1)$ norm of the solution, for any initial datum in $\Ltwozo$, decays as a function of time as $t\to\infty$. In addition, under more stringent initial regularity assumptions, we prove that the solution is continuous in $x$, hence cannot support any form of revivals.  
\item
For $0<\beta<1$,  we prove that the solution,  in addition to the same time decay property as the previous case, is of class $\ContinuousSpace^\infty[0,1]$ for all $t>0$, hence it is never discontinuous. Therefore the problem does not support any form of revivals.
  \item
For $\beta=1$, the problem supports  {\em weak cusp revivals}, as is proved in \cite{BFPS}. This means that the solution at the appropriate rational times can be expressed as the sum of two functions: a continuous  function, and a function which revives any initial discontinuity as a finite number of jump discontinuities and logarithmic cusps.

\end{itemize}

For the periodic problem for Airy on $(0,1)$,  with an initial condition $u_0\in\Lebesgue^2(0,1)$, 
the {\em periodic revival} property is completely  characterised by the solution formula at rational time $t=\frac p {(2\pi)^2}q$, given in \cite{erdougan2016dispersive} as
\be\label{perrev}
     u\Big(x,\frac{p}{4\pi^2q}\Big)= \frac{1}{q}
      \sum_{k,m=0}^{q-1}
     G_{p,q}(k,m) u_0\Big(x-2\pi\frac{k}{q}\Big),
\ee
where $G_{p,q}(k,m)\in\mathbb C$ are explicit coefficients. In particular,  $G_{p,1}$  is periodic in  $p$, hence we obtain the following relation, which justifies the term {\em revivals}:
\be\label{pper}
 u\Big(x,\frac{np}{(2\pi)^2}\Big)=G_{p,1}u_0(x),\qquad n\in{\mathbb N}.
 \ee

 The more general  notion of {\em weak revival} was first proposed in \cite{BFP}, and it allows the solution to have a continuous part as well as a periodic revival part.  {\em Cusp revivals} were first described  in \cite{boulton2020new} and then defined in greater generality in \cite{BFPS}, where the case $\beta=1$ is studied in detail. This  property implies that the solution revives any initial jump discontinuity, at the appropriate rational times, as finitely many jumps and logarithmic cusps.
The cusps are generated by the periodic Hilbert transform of a function that encodes the initial jump discontinuities.

Other examples of non-periodic problems are studied in  \cite{farmakis2023new,olver2018revivals}.

\section{Dirichlet-type boundary value problems for the Airy equation}
We consider the time evolution equation associated with the third-order operator
$$
T_\beta=(-i\partial_{x})^3:{\mathcal D_\beta}\longrightarrow \Ltwozo,\qquad \beta\in\R,
$$
where
\be
{\mathcal D_\beta}=\{v\in H^3(0,1):\;v(0)=v(1)=0,\; v_x(1)=\beta v_x(0)\}.
\label{airydom}
\ee
Noting that formally $u_{xxx}=-iT_\beta u$, we associate to this operator the following Dirichlet-type boundary value problem for the function $u(x,t)$:
\begin{align}
&u_t+u_{xxx}=0
\quad &&x\in(0,1), \;0<t<\infty, \nonumber
\\
&u(0,t)=u(1,t)=0,\quad u_x(1,t)=\beta u_x(0,t)
\quad&& 0<t<\tau,\label{pseD}\\
&u(\cdot,0)=u_0,\quad u_0\in \Ltwozo, && x\in(0,1).
\nonumber
\end{align}
The well-posedness of the problem in $\Ltwozo$ for $\beta\neq 0$, follows from general results in spectral theory, see \cite{DunfSchv2}.  In the case $\beta=0$, for which there is no eigenfunction basis, this problem is studied in \cite{pelloni2013spectral}.

\subsubsection*{Functional estimates and the range of $\beta$}
The real parameter $\beta$ gives the coupling of the values of first derivative at the end points of the interval. The cases of particular interest are $\beta=0$, $0<\beta<1$ and $\beta=1$, which we consider separately.  If $\beta=0$, the third boundary condition is
$u_x(1,t)=0$, and there is no coupling. If $0<\beta<1$ the boundary conditions are coupled, but in general the problem is not self-adjoint and indeed the associated eigenvalues are complex.    If $\beta=1$, so that $u_x(1,t)= u_x(0,t)$,  the problem is self-adjoint \cite{BFPS}.  We will review these three cases, that
illustrate the possible behaviour of any well-posed Dirichlet-type problem for this third order PDE.

Assume the equation admits a classical solution $u(x,t)$. Multiplying the equation by $2u$, we find after a formal integration by parts that
\be\label{fe1}
\|u(x,t)\|^2_{\Ltwozo}=\|u_0(x)\|^2_{\Ltwozo}-(1-\beta)\int_0^tu_x^2(0,s)ds .
\ee
Hence, if $\beta>1$, 
$$
\|u(x,t)\|^2_{\Ltwozo}\geq \|u_0(x)\|^2_{\Ltwozo},
$$
and the $\Lebesgue^2$ norm of the solution is not necessarily  bounded as a function of $t$, hence the problem may be ill-posed.%
\footnote{
    Indeed, it can be shown that for $\beta>1$ the problem is ill-posed, with the solution becoming unbounded for arbitrarily small times, \cite{smith2012,Smi2012b}.
}

Hence we assume that $\beta\leq 1$.
In this case, as is easily verified, both $T_\beta$ and its adjoint $T_\beta^*$ are  dissipative,  namely
$$
\langle T_\beta f,f\rangle\leq 0,\quad \langle T_\beta ^*f,f\rangle\leq 0\qquad \forall f\in\Ltwozo,
$$
and the associated boundary value problem  \eqref{pseD} admits a unique solution in $\Ltwozo$ \cite{bsz}.

This problem, specifically in operator-theoretic terms, has been studied in detail in \cite{pelloni2013spectral}. The study of the adjoint operator $T_\beta^\star$ is equivalent to reversing the direction of time and mapping $\beta\mapsto1/\beta$, i.e. in this case to considering the PDE
$
u_t-u_{xxx}=0$.
Because of this identification, it is enough to consider $\beta\in [0,1]$.\footnote{The case $\beta<0$ is entirely analogous, yielding a well -poseed problems for $-1\leq \beta<0$ for which similar results hold.}

It is straightforward to establish a  regularising estimate for the operator $T_\beta$, with $0\leq\beta<1$, a variation on, and slight generalisation of, the one given in \cite{bsz} for the case that $\beta=0$.\footnote{In \cite{bsz}, the equation studied includes the lower order term $u_x$.}

\begin{prop}\label{estprop}
Let $u_0\in \Ltwozo$ be given. For all $t\geq 0$,  and $0\leq \beta<1$, the solution of $u_t=iT_\beta u$ with $u(x,0)=u_0(x)$  satisfies

\begin{align}
&\|u(\cdot, t)\|_{\Ltwozo}^2\leq \|u_0\|_{\Ltwozo}^2,
\label{est0}\\
&\int_0^t u^2_x(0,s)ds\leq\frac 1 {1-\beta} \|u_0\|_{\Ltwozo}^2,
\label{est1}\\
&\int_0^1xu^2(x,t)dx+3\int_0^t\int_0^1u_x^2(x,s)dxds\leq \frac 1 {1-\beta}\|u_0\|_{\Ltwozo}^2.
\label{est2}\\
&\lim_{t\to\infty} \|u(\cdot, t)\|^2_{\Lebesgue^2(0,1)}=0.
\label{tlim}
\end{align}
\end{prop}
\begin{proof}
Assuming that $u$ is a strong (smooth) solution (a limiting argument, as in \cite{bsz} then gives the result for weak solutions), integration of $uu_t=-uu_{xxx}$ with the given boundary conditions yields~\eqref{fe1}, and, since $\beta<1$, this identity implies the two bounds \eqref{est0} and \eqref{est1}.

On the other hand, integration by parts of $2xuu_t=-2xuu_{xxx}$ using the given boundary conditions gives
\begin{align*}
\int_0^1xu^2(x,t)dx&+3\int_0^t\int_0^1u^2_x(x,s)dxds=\int_0^1xu_0^2(x)dx+\beta\int_0^tu_x^2(0,s)ds
 \\
&\leq \int_0^1xu_0^2(x)dx+\frac{\beta}{1-\beta}\int_0^1u_0^2(x)dx\leq \frac{1}{1-\beta}\|u_0\|^2_{\Ltwozo},
\end{align*}
which is \eqref{est2}.
By  \eqref{est1}, $f(t):=\|u(\cdot, t)\|_{\Lebesgue^2(0,1)}\in \Lebesgue^2(0,\infty)$. 
In addition, by integration by parts,  $f'(t)=- u^2_x(0,t)$, so that $f'(t)\in\Lebesgue^2(0,\infty)$ by \eqref{est1}.
This implies that  $\displaystyle{\lim_{t\to\infty}f(t)=0}$ so that \eqref{tlim}  must hold.
\end{proof}

These estimates imply that the function
$$
t\mapsto \|u(\cdot,t)\|_{H^1(0,1)}
$$
is in $\Lebesgue^2(0,t)$, so that for almost all $t>0$, $u(\cdot, t)\in H^1(0,1)$. In particular, for almost all $t>0$, the
solution admits a well  defined weak first derivative.  

\begin{remark} Heuristically,  fractalisation is a necessary consequence of the revival of initial discontinuities at every rational value of the time,  see \cite{mclau}. Therefore, if the  problem supported a form of revival, it would also present the dual property of fractalisation. This property would imply that the solution is nowhere differentiable,  which   cannot be the case given that  $u(\cdot, t)\in H^1(0,1)$ for almost all $t>0$. 
\end{remark}
This remark, as well as the time decay of the $L^2(0,1)$ norm of the solution,  suggests that for $0\leq \beta<1$, revivals are not supported. 

The main result of this paper is the precise  formulation  and proof of this statement.



\subsubsection*{The spectral structure of $T_\beta$}
Despite the formal similarity of the two problems obtained by choosing $\beta=0$ or $\beta\neq 0$, the respective spectral structure is fundamentally different.
Indeed, it can be shown that for the case $\beta=0$, although the operator $T_\beta$ has infinitely many discrete eigenvalues, the associated eigenfunctions do not form a Riesz basis, and that there is no generalised Fourier series representation for the solution of this problem~\cite{jackson, pelloni2005spectral, pelloni2013spectral}.

On the other hand, if $0<\beta\leq 1$, there is an infinite family of discrete, generally complex, eigenvalues, and the associated eigenfunctions and their adjoints form a complete biorthogonal basis. Hence it is possible to represent the solution of the boundary value problem as  a generalised Fourier series.

The case $\beta=1$ is the only one resulting in a self-adjoint problem, for which the eigenvalues are purely real.  On the basis of this spectral structure, we distinguish three separate cases: 
\begin{itemize}
\item[(a)]
$0<\beta<1$: there is a generalised Fourier series representation for the solution. We prove that the solution is smooth in $x$ and its $L^2(0,1)$ norm decays in $t$. Hence no revival property holds. 
 \item[(b)]$\beta=0$: there is no generalised Fourier series representation for the solution. The  $L^2(0,1)$ norm of the solution decays in $t$ and, under slightly higher regularity assumptions, we prove that the solution is continuous, hence no revival property can hold.
 \item[(c)]
$\beta=1$: there is a generalised Fourier series representation for the solution, and a weak cusp revival property holds, as proved in \cite{BFPS}.
\end{itemize}


\subsection*{The case $0<\beta<1$}

We prove in Theorem \ref{thm2} that, for any integrable initial condition, in addition to the time decay of its $L^2(0,1)$, the solution is actually of class $\ContinuousSpace^\infty$, for all $t>0$. This implies that this problem cannot support any type of revivals. Indeed, any form of the revival property implies that the solution is not continuous at countably many values of $t$.

\subsubsection*{Spectral structure}
If $\beta\neq 0$, the operator is associated to a biorthogonal basis for $\Ltwozo$,
composed of the eigenfunctions of $T_\beta$ and those of the adjoint $T_\beta^\star$ \cite{DunfSchv2,jackson}. This statement holds provided all eigenvalues are simple, which we henceforth assume to be the case, based on the argument given below that shows that all but finitely many of them are simple.
Hence we can study the generalised Fourier series representation of the solution with a view to verify the conjecture that the revival property does not hold in this case.

As shown in \cite{pelloni2005spectral}, the eigenvalues are given by $k_n^3$, where the $k_n$'s are the zeros of
$$
\Delta(k)=\beta[\re^{-ik}+\alpha \re^{-i\alpha k}+\alpha^2\re^{-i\alpha^2 k}]+\re^{ik}+\alpha \re^{i\alpha k}+\alpha^2\re^{i\alpha^2 k}, \quad \alpha=\re^{2\pi i/3}.
$$

While this transcendental equation cannot be solved exactly,  
the position of finitely many $k_n$ can be found numerically using an algorithm based on the argument principle.
For example, if $\beta=10^{-6}$, then the first few lie at the crosses in figure~\ref{fig:delta-roots-locus}.
\begin{figure}[ht]
    \centering
    \includegraphics[width=0.8\linewidth]{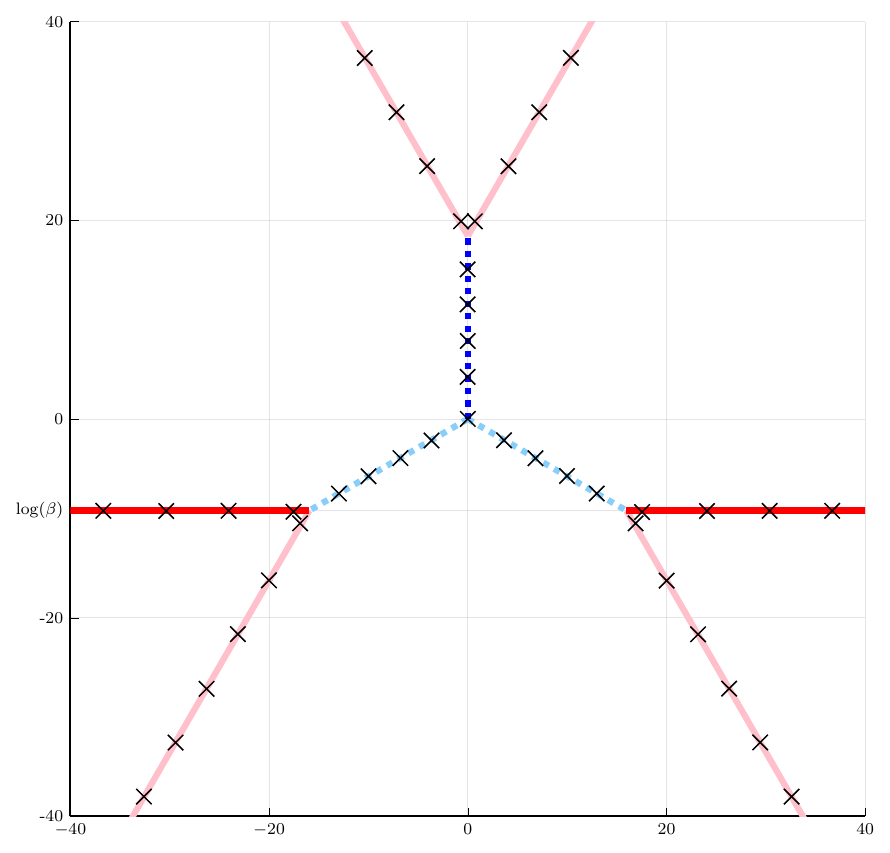}
    \caption{
        The zeros of $\Delta$ for $\beta=10^{-6}$.
        Also shown are line segments and rays along which the zeros lie, approximately.
    }
    \label{fig:delta-roots-locus}
\end{figure}
It appears that they lie approximately along the dashed blue line segment and the solid red rays, and their lighter colored rotations about the origin by $2\pi/3$, $4\pi/3$.
The rotational symmetry is immediately confirmed by the identity $\Delta(\alpha k) = \alpha^2\Delta(k)$, and implies that we only need to characterise the location of the solid red horizontal rays and dashed blue vertical line segment.
To this end, following the approach of~\cite{langer},
consider $k$ with $\Re(k)\gg\lvert\Im(k)\rvert$, and observe that the dominant terms in $\Delta(k)$ are $\beta\alpha e^{-i\alpha k}$ and $\alpha^2 e^{i\alpha^2k}$, while all other terms are $\mathcal{O}(1)$ in $\Re(k)$.
Imposing that the leading order terms approximately cancel requires that $k_n\approx \left(2n-\frac13\right)\pi + i \log(\beta)$ for  $n\to\infty$.
The existence of zeros close to $-\left(2n-\frac13\right)\pi + i \log(\beta)$ is similarly justified. Note that since the zeros of $\Delta'$ are asymptotically located at $2n\pi + i\log(\beta)$, all but finitely many of the zeros of $\Delta$ must be simple. 
The zeros close to the dashed blue line segment arise when $\beta$ is very small so that, for $\lvert k \rvert$ not too large, the terms in $\Delta(k)$ with factor $\beta$ are negligible.
Using the same kind of asymptotic argument as above, choosing a region of $\mathbb{C}$ such that two terms dominate the remaining one and requiring that those leading order terms cancel, we find that there is a finite number of zeros around $i\left(2n-1\right)\frac\pi{\sqrt3}$, for $n$ positive but not too large.
Note that, although $0$ is a zero of $\Delta$, it is not the cube root of an eigenvalue, because no quadratic expression satisfies the boundary conditions.

The above arguments show that on the horizontal (red) rays, 
\begin{align}
    k_n = \kappa_n + i\log\beta + \epsilon_n, &\qquad \mbox{where, as } |n|\to\infty,\label{asymp1}\\
    \kappa_n = \pi\left( 2n - \frac{\operatorname{sgn}(n)}{3} \right), &\qquad
    \epsilon_n = \mathcal{O}\left(e^{-\sqrt{3}\pi\lvert n \rvert}\right).\nonumber
\end{align}
The corresponding eigenfunctions are given by
\be\label{ev1}
    X_n = \sum_{j=0}^2 e^{-i\alpha^jk_nx}(e^{-i\alpha^{j+1}k_n} - e^{-i\alpha^{j+2}k_n}).
\ee
It is not difficult to show that  the adjoint operator has eigenvalues $k_n^\star = \overline{k_n}$ and eigenfunctions $X_n^\star$ given by the same formula as above for $X_n$ but with $k_n^\star$ replacing $k_n$.

\subsubsection*{General solution formula}
We can express the solution of this boundary value problem as a generalised Fourier series, namely
\be\label{sol1rep}
    u(x,t) = \sum_n e^{-ik_n^3t} X_n(x) \frac{\langle u_0 , X_n^\star \rangle}{\langle X_n , X_n^\star \rangle}.
\ee

We can now prove the following regularity result for the solution, which in particular confirms that no revival property can be supported by boundary value problems in this class. 

\begin{theorem}\label{thm2}
    Suppose $u_0\in \Ltwozo$.
    Then, for all $t >0$, the solution $u$ of the problem \eqref{pseD} with $0<
\beta<1$ satisfies $u({}\cdot{},t)\in \ContinuousSpace^\infty[0,1]$.
\end{theorem}
\begin{proof}
Using  the formula \eqref{ev1} for the eigenfunctions, we can compute
\begin{align*}
    \left\langle X_n , X_n^\star \right\rangle
    \hspace{-4em} &\hspace{4em}= \sum_{j,r=0}^2 \int_0^1 e^{-i\alpha^jk_nx} \overline{e^{-i\alpha^r\overline{k_n}x}} dx \left( e^{-i\alpha^{j+1}k_n}-e^{-i\alpha^{j+2}k_n} \right) \overline{\left( e^{-i\alpha^{r+1}\overline{k_n}}-e^{-i\alpha^{r+2}\overline{k_n}} \right)} \\
    &= \sum_{j=0}^2 \left( e^{-i\alpha^{j+1}k_n}-e^{-i\alpha^{j+2}k_n} \right) \left( e^{i\alpha^{j+2}k_n}-e^{i\alpha^{j+1}k_n} \right)
    \\
    &\hspace{1em} + \sum_{\substack{j,r=0: \\ \alpha^j \neq \alpha^{-r}}}^2 \left[ \frac{e^{-ik_n(\alpha^j-\alpha^{-r})}-1}{-ik_n(\alpha^j-\alpha^{-r})} \right] \left( e^{-i\alpha^{j+1}k_n}-e^{-i\alpha^{j+2}k_n} \right) \left( e^{i\alpha^{-(r+1)}k_n}-e^{i\alpha^{-(r+2)}k_n} \right).
\end{align*}
Recalling expression \eqref{asymp1} for the asymptotic behaviour of the eigenvalues,  we find that the first sum, as $n\to\pm\infty$,  is  $e^{\sqrt3\lvert k_n \rvert} + \mathcal{O}(e^{\sqrt3\pi\lvert n \rvert})$.

The terms in the second sum each take one of the forms
$$
    \frac{e^{-ik_n(\alpha^\ell-\alpha^m)}}{-ik_n(\alpha^j-\alpha^{-r})}, \qquad \frac{e^{-ik_n(\alpha^j(1+\alpha^\ell)-\alpha^{-r}(1+\alpha^m))}}{-ik_n(\alpha^j-\alpha^{-r})},
$$
for $\ell,m\in\{0,1,2\}$ in the first form and $\ell,m\in\{1,2\}$ in the second form.
Therefore, these terms are all dominated by the leading order term from the first sum.
In summary, we obtain
$$
    \left\langle X_n , X_n^\star \right\rangle = e^{\sqrt3\lvert k_n \rvert} + \mathcal{O}\left(\frac{e^{2\sqrt3\pi\lvert n \rvert}}n\right).
$$
By H\"older's inequality,
$$
    \langle u_0 , X_n^\star \rangle
    \leq \left\lVert u_0 \right\rVert_{\Lebesgue^1(0,1)} \left\lVert X_n^\star \right\rVert_{\Lebesgue^\infty(0,1)}
    \lesssim \left\lVert u_0 \right\rVert_{\Lebesgue^1(0,1)} e^{\frac{\sqrt3}2\lvert \kappa_n \rvert},
$$
with constant uniform in $n$.
Therefore
$$
    \left\lVert \frac{\left\langle u_0 , X_N^\star\right\rangle}{\left\langle X_n , X_n^\star\right\rangle} X_n \right\rVert_{\Lebesgue^\infty(0,1)}
$$
is bounded as $n\to\pm\infty$ and, for the $m$-th derivative, we obtain
$$
    \left\lVert \frac{\left\langle u_0 , X_N^\star\right\rangle}{\left\langle X_n , X_n^\star\right\rangle} X_n^{(m)} \right\rVert_{\Lebesgue^\infty(0,1)} = \mathcal{O}\left( n^m \right).
$$
On the other hand, computing the imaginary part of the eigenvalues yields\
\begin{align*}
    \Im(k_n^3) &= \Im\left( \kappa_n^3 + 3i\log(\beta)\kappa_n^2 - 3[\log(\beta)]^2\kappa_n - i[\log(\beta)]^3 + \mathcal{O}\left(n^2e^{-\sqrt3\pi n}\right) \right)
    \\&= 3\log(\beta)\kappa_n^2 + \mathcal{O}(1).
\end{align*}
It follows that
$$
    e^{-ik_n^3t} = \mathcal{O}\left(e^{3\kappa_n^2t\log(\beta)}\right)
$$
and, since $t\log(\beta)<0$ for $0<\beta<1$, this factor is decaying exponentially.
Hence, for all $t>0$, the series in equation~\eqref{sol1rep} for $u(x,t)$ converges absolutely uniformly in $x\in[0,1]$, and so do its term-by-term spatial derivatives.
From the fact that  $X_n\in \ContinuousSpace^\infty[0,1]$, it follows that, for all $t>0$, $u({}\cdot{},t)\in \ContinuousSpace^\infty[0,1]$.
\end{proof}

\subsection*{The case $\beta=0$}

For this case, we know that the operator has infinitely many real eigenvalues, but that the eigenfunctions do not form a Riesz basis for the associated $\Lebesgue^2$ space \cite{pelloni2013spectral}.

This boundary value problem was studied in~\cite{bsz}, where a version of Proposition \ref{estprop}  is proved. In that paper, it is also reported how, when the initial datum is in the domain of the operator,
$u_0\in {\mathcal D_0}$,  there exist  a unique classical solution $u\in \ContinuousSpace([0,\infty); H^3(0,1))\cap \ContinuousSpace^1([0,\infty); \Ltwozo)$ of the boundary value problem.
When $u_0\in \Ltwozo$, there is a unique solution in the same class.

\subsubsection*{Decay in time}

 We now show how the assumption of revivals leads to a contradiction. This contradiction is reached if  the solution is assumed to satisfy a specific form of weak revival property, including an additional time decay property of its  continuous component. 
 
 \medskip
 More precisely, suppose that the solution $u(x,t)$  of  \eqref{pseD} with $\beta=0$ has the following weak revival property:

\smallskip
\noindent
{\em 
At rational times  $t=\frac {p}{(2\pi)^2}$, $u(x,t)$  can be written as $u(x,t)=U_{\mathcal R}(x,t)+U_{\mathcal C}(x,t)$, where $U_{\mathcal C}(x,t)$ is continuous and decays as $t\to\infty$, while 
$U_{\mathcal R}(x,t)$ satisfies the periodic revival property \eqref{perrev}, possibly with $u_0$ modified by multiplication with a $C^\infty$ function (this does not change the location and number of the initial jump discontinuities). 
}
  
\smallskip
The decay of $\|u(\argdot,t)\|_{\Lebesgue^2(0,1)}$ as $t\to\infty$, given by proposition \ref{estprop},  implies that no such weak revival property can hold.  

Indeed, if it did, then for each fixed $p\in\mathbb N$ there would be arbitrarily large times $t_n=\frac{np}{(2\pi)^2}$, $n\in\N$, where the revival part of solution $U_{\mathcal R}(x,t)$, according to \eqref{pper},   revives  the initial condition, up to a fixed multiple, hence for infinitely many values of the time,  its norm is equal to a multiple of $\|u_0\|_{\Lebesgue^2(0,1)}$.  Therefore it cannot hold that $\displaystyle{\|U_{\mathcal R}(\argdot,t)\|_{\Lebesgue^2(0,1)}\to_{t\to\infty}0}$.

In particular, the solution of this problem cannot have the periodic revival property ($U_{\mathcal C}(x,t)=0$) 
or a weak version of the type described in  theorem \ref{lkdvthm} below for the case $\beta=1.$  Indeed, in the latter case, the continuous part of the solution is an exponentially decaying function of $t$ as $t\to\infty$. 
%
%
%
%
%

\subsubsection*{Continuity properties of the solution}
Using the unified transform method of Fokas \cite{ABS2022a,deconinck2014method,fokas2005transform, smithfokas}, it can be shown that the  boundary value problem corresponding to $\beta=0$, for a smooth $u_0$, has a unique solution $u(x,t)$, and that solution admits the explicit  contour integral representation
\begin{multline} \label{eqn:soln.beta0}
    u(x,t) = \frac 1 {2\pi}\int_{\mathbb R} \re^{ikx+ik^3t}\widehat{u_0}(k)dk \\
    + \left\{\int_{\Gamma^+}\re^{ikx+ik^3t}\frac {\zeta^+(k)}{\Delta(k)}dk+
    \int_{\Gamma^-}\re^{ik(x-1)+ik^3t}\frac {\zeta^-(k)}{\Delta(k)}dk\right\}.
\end{multline}
In this expression,  $\Gamma^{\pm}$ are contours in $\C^\pm$, defined as the lines where the exponential $\re^{ik^3t}$ is purely oscillatory. Namely, the contours $\Gamma^\pm$ are the positively oriented boundaries of sectors $k\in\C^\pm$ with $\Re(ik^3)>0$, 
    \begin{align}\label{Delta} 
       & \Delta(k)  = \re^{-ik} + \alpha\re^{-i\alpha k} + \alpha^2\re^{-i\alpha^2k},\\
       & \zeta^-(k) = \widehat{u_0}(k) + \alpha\widehat{u_0}(\alpha k) + \alpha^2\widehat{u_0}(\alpha^2k), \label{zeta-}\\
       & \zeta^+(k) = \re^{-ik}\zeta^-(k) - \widehat{u_0}(k)\Delta(k) \label{zeta+}\\
        &= \alpha \left( \re^{-ik} \widehat{u_0}(\alpha k) - \re^{-i\alpha k} \widehat{u_0}(k) \right) + \alpha^2 \left( \re^{-ik} \widehat{u_0}(\alpha^2k) - \re^{-i\alpha^2k} \widehat{u_0}(k) \right).\nonumber
    \end{align}
The zeros of $\Delta$  are the cube roots of the eigenvalues of the spatial differential operator, and it can be shown (as we detail in the proof of Theorem \ref{thm:beta0.continuous} below) that they are not on the integration contours. Hence the integrands  in the last two integrals are analytic functions on the contours, and the integrals are well defined.

This integral representation is not equivalent to a discrete generalised Fourier series. Nevertheless, we are able to establish  continuity of $u(\cdot, t)$ for all $t>0$, provided the initial condition, while still discontinuous, has specific piecewise regularity. The continuity of $u$ as a function of $x$ implies that no revival property can be  supported.

\begin{theorem} \label{thm:beta0.continuous}
    Suppose that $u_0\in\Ctwozo$ piecewise.
    Let $u(x,t)$ be defined by equation~\eqref{eqn:soln.beta0}.
    Then, for all $t>0$, $u(\argdot,t)\in\Czopen$.
\end{theorem}

\begin{proof}[Proof of Theorem~\ref{thm:beta0.continuous}]
    We begin by arguing that some components of the contour integrals in equation~\eqref{eqn:soln.beta0} yield, for all $t>0$, functions $u(\argdot,t)$ which are $\ContinuousSpace^\infty[0,1]$.
    The final part of the proof, given as a separate proposition,  justifies continuity of the remaining components.

    As  shown in~\cite{pelloni2005spectral}, the zeros of $\Delta$ lie exactly on the rays $-i\alpha^j\R^+$ and asymptotically at $-i\alpha^j(2n+\frac13)\pi/\sqrt3$ for $n\in\N$ and $j\in\{0,1,2\}$, with exponentially decaying error.
    Moreover, because the zeros of $\Delta'$ lie asymptotically at $-i\alpha^j(2n+\frac13)\pi/\sqrt3$, at most finitely many zeros of $\Delta$ are nonsimple.
    Careful bounds can be obtained to show that all nonzero zeros are simple.
    There is a removable singularity of $\zeta^\pm/\Delta$ at $0$.
    We denote the nonzero zeros of $\Delta$ lying along $-i\R^+$ by $k_n$ for $n\in\N$.
    
    Consider $k\to\infty$ with $-2\pi/3 \leq \arg(k) \leq -\pi/3$.
    Away from zeros of $\Delta$,
    \[
        \Delta(k) = \re^{i\frac k2} \left( \alpha \re^{-i\frac{\sqrt3 ik}{2}} + \alpha^2 \re^{i\frac{\sqrt3 ik}{2}} \right) + \bigoh{\re^{-\frac{\abs{k}\sqrt3}2}}.
    \]
    Let $\epsilon>0$ be sufficiently small that both $u_0\rvert_{[0,\epsilon]}$ and $u_0\rvert_{[1-\epsilon,1]}$ are $\ContinuousSpace^1$ functions.
    Integrating by parts, for $k$ away from the zeros of $\Delta$,
    \begin{align*}
        \alpha^j\widehat{u_0}(\alpha^jk)
        &=
        \frac ik \left(
            \left[\re^{-i\alpha^jkz}u_0(z)\right]_{z=0}^{z=\epsilon}
            + \left[\re^{-i\alpha^jkz}u_0(z)\right]_{z=1-\epsilon}^{z=1} \right. \\
            &\hspace{4em} - \left.\left\{ \int_0^\epsilon + \int_{1-\epsilon}^1 \right\} \re^{-i\alpha^jkz} u^\prime_0(z)\D z
        \right)
        + \alpha^j \int_\epsilon^{1-\epsilon}\re^{-i\alpha^jkz} u_0(z) \D z \\
        &= \bigoh{k^{-1}\Delta(k)},
    \end{align*}
    so also is $\zeta^-(k)$.
    Define the sectors
    \[
        E^\pm = \{k\in\C^\pm \mbox{ such that } \Re(ik^3)<0 \}.
    \]
    By the asymptotic estimates established above, Jordan's lemma, and a residue calculation,
    \begin{equation} \label{eqn:beta0.E-IsSmooth}
        \int_{\partial E^-} \re^{ik(x-1)+ik^3t} \frac{\zeta^-(k)}{\Delta(k)} \D k
        =
        \sum_{n\in\N} \re^{ik_n(x-1)+ik_n^3t} \frac{\zeta^-(k_n)}{\Delta'(k_n)}.
    \end{equation}
    From the asymptotic formula for $k_n$, 
    \(
        \Delta'(k_n) = (-1)^{n+1} i \re^{(2n+\frac13)\frac\pi{2\sqrt3}} + o(1),
    \)
    so $\frac{\D^m}{\D k^m}\exp(ik_n(x-1))\zeta^-(k_n)/\Delta(k_n)=\bigoh{n^{m-1}}$ as $n\to\infty$, uniformly in $x\in[0,1]$.
    Because, for all $t>0$, $\re^{ik_n^3t}$ decays exponentially as $n\to\infty$, the series in equation~\eqref{eqn:beta0.E-IsSmooth} converges absolutely uniformly in $x$, along with all its termwise derivatives.
    Hence the integral in equation~\eqref{eqn:beta0.E-IsSmooth} represents a $\ContinuousSpace^\infty[0,1]$ function of $x$.
    Consequently, equation~\eqref{eqn:soln.beta0} may be reexpressed as
    \[
        2\pi u(x,t) = V_{\mathcal{C}}(x,t) - \int_{\partial E^+} \re^{ikx+ik^3t}\frac{\zeta^+(k)}{\Delta(k)} \D k,
    \]
    in which $V_{\mathcal{C}}$ is, for all $t>0$, a $\ContinuousSpace^\infty$ function of $x\in[0,1]$.
    
    Let
    \[
        \rho(k) = \re^{-ik} \widehat{u_0}(\alpha^2k) - \re^{-i\alpha^2k} \widehat{u_0}(k),
    \]
    so that $\zeta^+(k) = \alpha^2\rho(k) - \alpha\rho(\alpha k)$.
    The ratio $\rho(\alpha k)/\Delta(k)$ decays exponentially as $k\to\infty$ within the sector $\frac{2\pi}3 \leq \arg(k) \leq \pi$ away from the zeros of $\Delta$, as does $\rho(k)/\Delta(k)$ within the sector $0 \leq \arg(k) \leq \frac\pi3$.
    Hence, by a similar argument to above%
    ,
    for all $t>0$, these terms yield contributions to $u(x,t)$ that are $\ContinuousSpace^\infty[0,1]$ in $x$.
    Changing variables, so that the contour about the two connected components of $E^+$ is replaced with the boundary of the single sector $E_2^+=\{k\in\C$ with $2\pi/3<\arg(k)<\pi\}$, we find that
    \be\label{ucont12}
        2\pi u(x,t) = U_{\mathcal{C}}(x,t) - \int_{\partial E^+_2} \re^{ik^3t}\frac{\rho(k)}{\Delta(k)} \left( \alpha^2\re^{ikx} - \alpha\re^{i\alpha^2kx} \right) \D k,
    \ee
    in which $U_{\mathcal{C}}$ is, for all $t>0$, a $\ContinuousSpace^\infty$ function of $x\in[0,1]$.
 
To complete the proof of the theorem,  it remains to show that, for all $x\in(0,1)$, $t>0$, $\lim_{\delta\to0}u(x,t)-u(x+\delta,t)=0$.  This is the conclusion of Proposition \ref{contprop},  proved by a careful and lengthy computation, given in appendix. 
\end{proof}

It is an immediate corollary of Theorem~\ref{thm:beta0.continuous} that the boundary value problem with $\beta=0$ does not support revivals.

\subsection*{The case $\beta=1$}
In this case,we need to analyse the spectral structure of the operator
$$
T_1=-i\partial_{x}^3:{\mathcal D_1}\longrightarrow \Ltwozo.
$$
It follows from general spectral theory that $A$ is self-adjoint, hence that $iA$ is the generator of a $C_0$~one-parameter semigroup.
This guarantees, for any $u_0\in \Ltwozo$, the existence of a unique solution $u(x,t)\in \ContinuousSpace([0,\infty); \Ltwozo)$ of the associated boundary value problem \eqref{pseD}.  The problem is  conservative, so the $L^2(0,1)$ norm of the solution does not decay in time.

If $u_0\in {\mathcal D_1}$,  then $u$ is a strong solution: $u\in {\mathcal D_1}$ and
$$
u\in C([0,\infty); H^3(0,1))\cap \ContinuousSpace^1([0,\infty); \Ltwozo).
$$

We are interested in
the case when $u_0$ has some jumps discontinuities, but is otherwise smooth.
While this is the case of most interest, the theorem below is proved  in \cite{BFPS} for a more general class of initial data.
\begin{theorem}\label{lkdvthm}
    Assume that $u_0\in \BVzo$ is piecewise Lipschitz. 
    Then, the solution of the boundary value problem~\eqref{pseD}, with $\beta=1$,  admits the representation
    $$
        u(x,t)=U_{\mathcal R}(x,t)+U_{\mathcal C}(x,t), \qquad x\in(0,1),\;t>0,
    $$
    where
   
    $U_{\mathcal C}(\argdot,t)\in\Czopen$ for all $t\in\R$, and
    \begin{equation} \label{Ldef0}
        U_{\mathcal R}(x,t)=\sum_{n=1}^\infty 2 \operatorname{Re} \left[\widehat{G_{u_0}}(n)
        \re^{i(2n-\frac 1 3 )^3\pi^3t}\re^{i (2n-\frac 1 3 )\pi x} \right],
    \end{equation}
    where

At any time of the form   $t=\frac p{\pi^2 q}$, with $p,q\in\N$ coprime, the function  $U_{\mathcal R}(x,t)$ can be written as
    \begin{equation} \label{finrevf2NEW}
        U_{\mathcal R}\left(x,\frac p {q \pi^2}\right) = \Re\left\{ \re^{-i\frac{\pi}{27}\left(9x+\frac pq\right)} [L_1(x)+iL_2(x)-L_3] \right\},
    \end{equation}
    where
    \begin{align*}
        L_1(x)
        &= \sum_{k=0}^{q-1} d_k^{p,q}
        G_{u_0}\left(x+\frac p{3q}-\frac{k}{q}\right),
        \\
        L_2(x)
        &= \sum_{k=0}^{q-1} d_k^{p,q}
        [{\mathcal H}G_{u_0}]\left(x+\frac p{3q}-\frac{k}{q}\right),
        \\
        L_3
        &=
            \int_0^1G_{u_0}(y)\mathrm{d}y,
    \end{align*}
$\mathcal{H}$ denotes the $1$-periodic Hilbert transform,  and
 \be\label{dkpq}
        d_k^{p,q} = \frac1q \sum_{m=0}^{q-1} \re^{2\pi i\left( mk + \left[4m^3-2m^2\right]p \right)/q}.
    \ee
    \end{theorem}
The conclusion of this theorem is that, up to a continuous perturbation,  the solution of \eqref{pseD} with $\beta=1$ at rational times $t=\frac p{\pi^2 q}$ is given by a finite superposition of copies of the initial datum
and of the 1-periodic Hilbert transform of the initial datum.
Since the periodic Hilbert transform maps every jump discontinuity to a logarithmic cusp, the profile of the solution at rational times, when the initial datum has jump discontinuities, is the superposition of a  continuous function with a finite number of jumps and logarithmic cusps, namely it has the \emph{weak cusp revival property} discussed in~\cite{BFPS}.

\subsection*{Numerical simulations}

\begin{figure}[ht]
  \centering
    \includegraphics[width=1\linewidth]{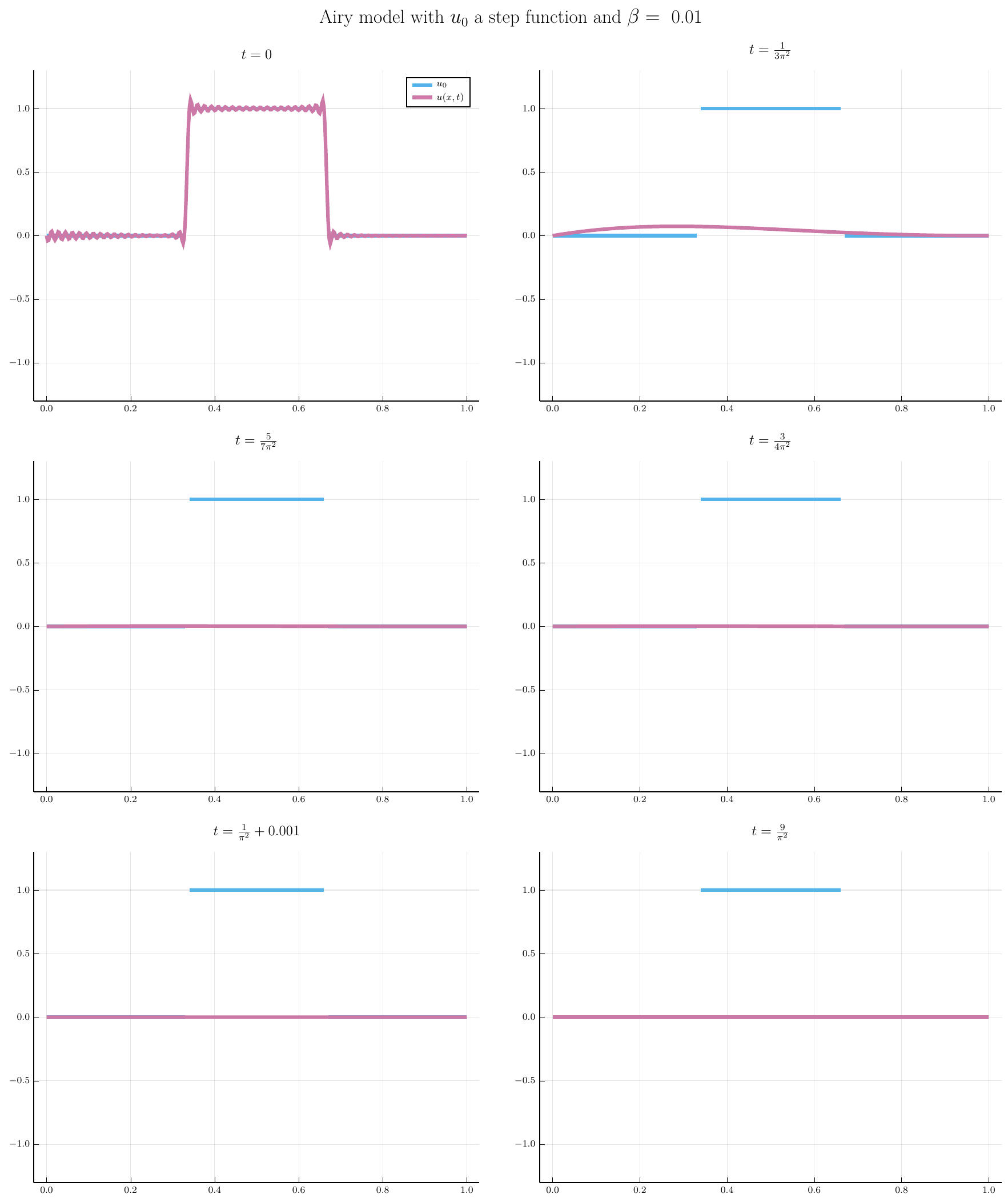}
  \caption{}
    \label{fig001}
   \end{figure}
   
\begin{figure}[ht]
  \centering
    \includegraphics[width=1\linewidth]{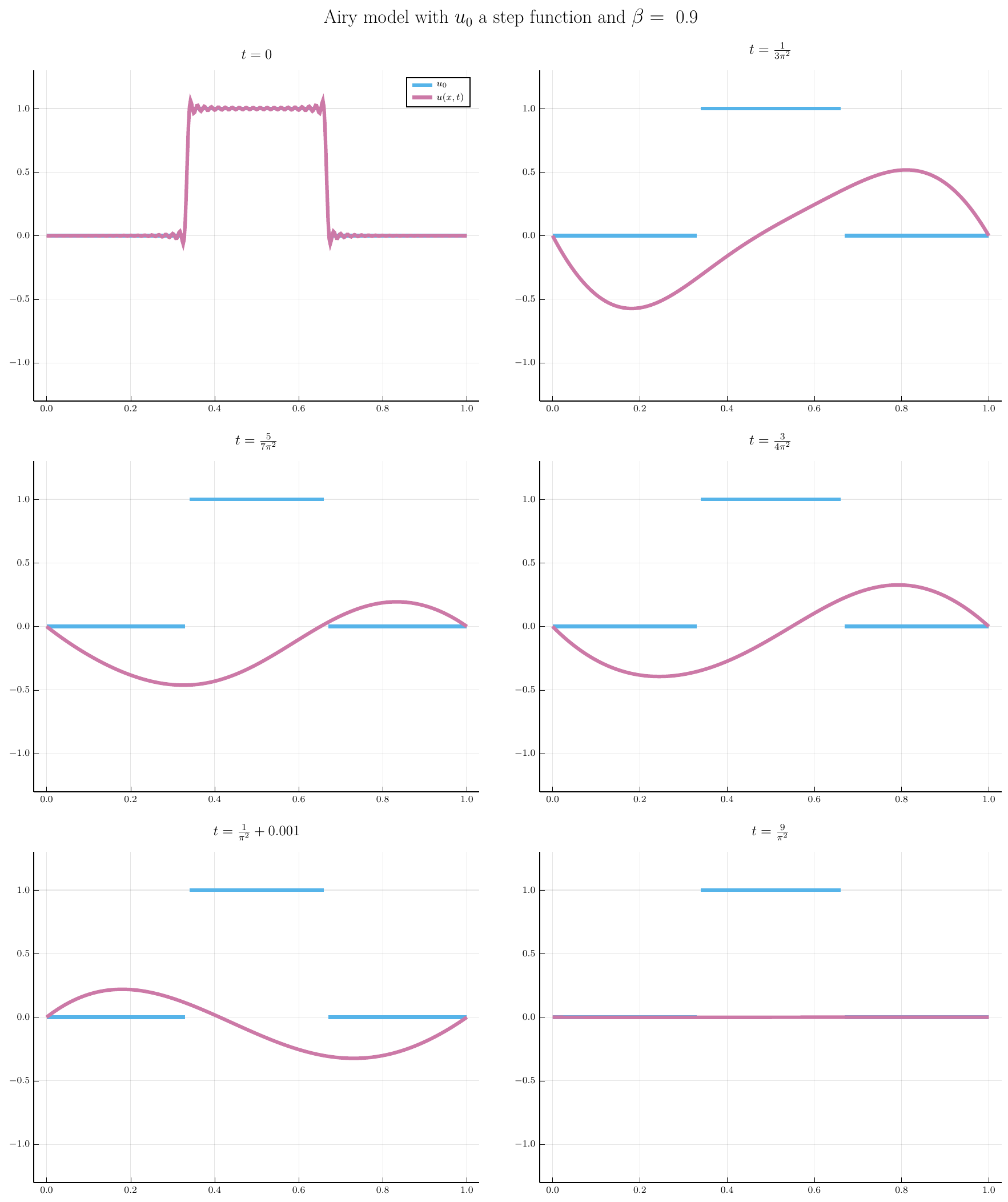}
  \caption{}
    \label{fig09}
   \end{figure}
 
\begin{figure}[ht]
  \centering
    \includegraphics[width=1\linewidth]{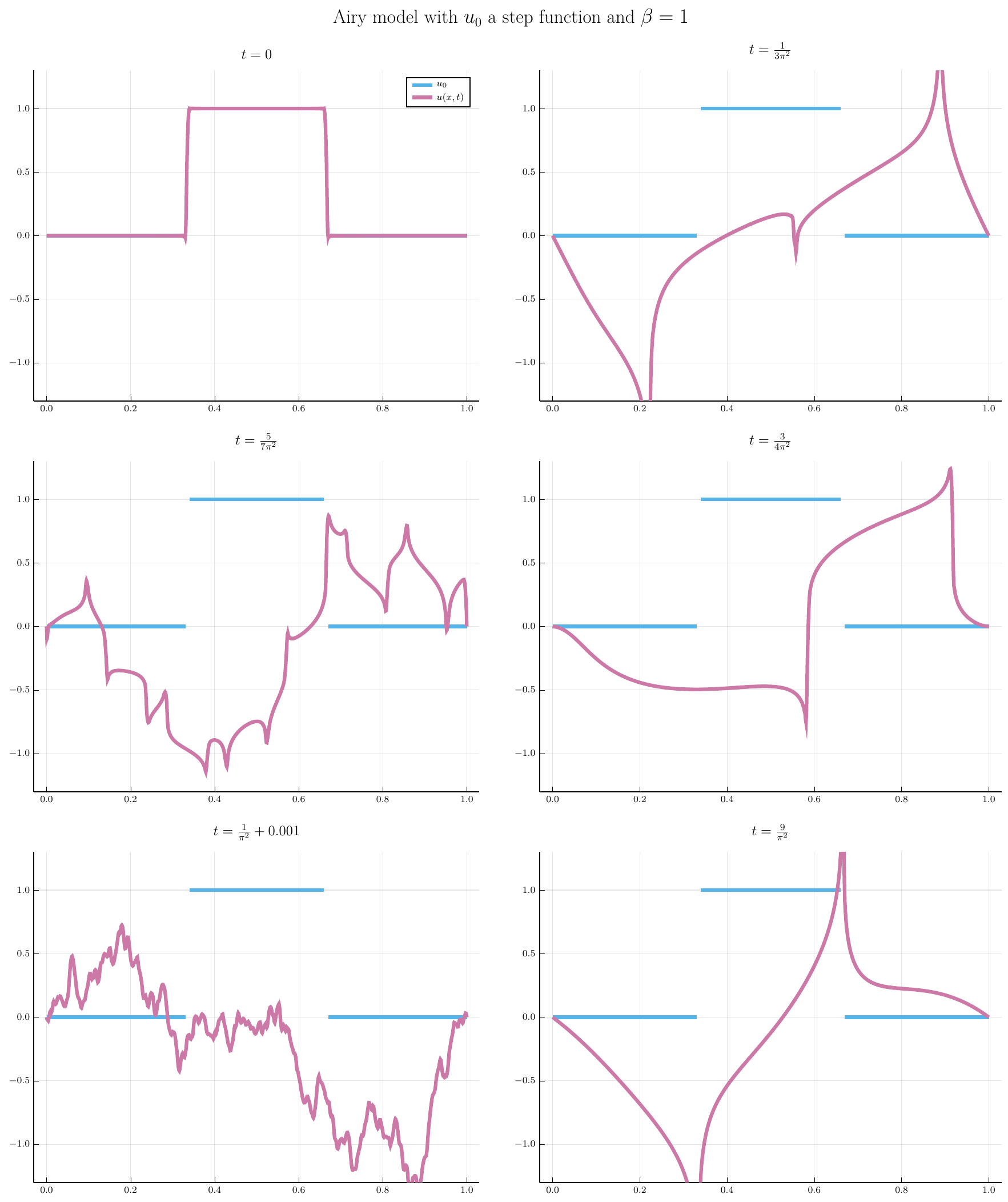}
  \caption{}
    \label{fig1}
   \end{figure}
   
We used Julia to plot the solution, approximated by the truncated eigenfunction series.
In the $\beta=1$ case, because the problem is selfadjoint, we used the ApproxFun package~\cite{ApproxFun.jl-2014} to compute the eigenvalues and eigenfunctions directly and evaluate the inner products in the solution representation.
When $0<\beta<1$, we instead calculate the $k_n$ as roots of $\Delta$, then use expression~\eqref{ev1} for the eigenfunctions.
For $k_n$ close to $0$, we use the Julia implementation~\cite{Gasdia19} of the GRPF algorithm~\cite{Kow2015a,Kow2018a}, based on the argument principle and, for $\lvert k_n \rvert \gg0$, we use the leading order asymptotic term in the expression for $k_n$ given in equation~\eqref{asymp1}.
We use the first 614 eigenvalues for the selfadjoint problem where $\beta=1$ and the first 120 for problems with $0<\beta<1$.
In the case $\beta=0$, as may be inferred from the delicacy of the asymptotic analysis in the proof of theorem~\ref{thm:beta0.continuous}, numerical evaluation of the solution representation via highly oscillatory contour integrals~\eqref{eqn:soln.beta0} is somewhat delicate, so is reserved for future study.

A simplified version of the GPRF algorithm to find all roots of an analytic function within a bounded region proceeds as follows.
Select sample points on the region's boundary, evaluate the principal value of the argument of the function at each sample point.
Proceeding anticlockwise about the boundary, calculate the difference between the arguments at adjacent sample points and count the number times this jumps from $\pi^-$ to $-\pi^+$, subtracting any jumps in the opposite direction.
By the argument principal, this is equal to the number of zeros interior to the region.
Now bisect the polygon, and repeat on each subregion.
Excluding from further consideration any subregions containing no roots, recursively subdivide and count roots to the desired accuracy.

%
%

We present plots of the solution starting from a piecewise constant initial datum for both cases. Specifically, in these plots the initial datum is the function
$$
u_0(x)=\begin{cases}
1&\frac13<x<\frac23,
\\
0& \mbox{otherwise}.
\end{cases}
$$
The plots for $\beta=0.01$ and $\beta=0.9$ clearly show  the solution starting from this discontinuous initial datum to be  smooth (and decaying) at all positive times.
The plots for $\beta=1$ show the cusps (which are infinite logarithmic cusps, though they are shown as finite by the graphic interface)  generated by the discontinuities at rational times, and the fractalisation phenomenon  at the generic time
$\displaystyle{t=\frac 1 {\pi^2}+0.001}$.

\section*{Acknowledgements}
BP was partially supported by a Leverhulme Research Fellowship. DAS gratefully acknowledges support from the Quarterly Journal of Mechanics and Applied Mathematics Fund for Applied Mathematics.

\section*{Appendix}

\begin{prop}\label{contprop}
Let $u(x,t)$ be defined by equation \eqref{ucont12}. 
Then, for all $x\in(0,1)$ and all $t>0$,
$$
\lim_{\delta\to 0}u(x+\delta,t)=u(x,t).
$$
\end{prop}

\begin{proof}
To prove the claim, we need to  evaluate the $\delta\to0$ limit of
    \[
        \int_{\partial E^+_2} \re^{ik^3t}\frac{\rho(k)}{\Delta(k)} \left( \alpha^2\re^{ikx}\left[1-\re^{ik\delta}\right] - \alpha\re^{i\alpha^2kx}\left[1-\re^{ik\alpha^2\delta}\right] \right) \D k.
    \]
    Consider the integral along the negative real axis, which can be rewritten as $J_1+\alpha^2J_2-\alpha^2J_3-\alpha J_4$, for
    \begin{align*}
        J_1 &= \int_0^{\abs\delta^{\frac{-1}3}} \re^{-ik^3t}\frac{\rho(-k)}{\Delta(-k)} \left( \alpha^2\re^{-ikx}\left[1-\re^{-ik\delta}\right] - \alpha\re^{-i\alpha^2kx}\left[1-\re^{-ik\alpha^2\delta}\right] \right) \D k, \\
        J_2 &= \int_{\abs\delta^{\frac{-1}3}}^\infty \re^{-ikx-ik^3t}\frac{\rho(-k)}{\Delta(-k)} \D k, \\
        J_3 &= \int_{\abs\delta^{\frac{-1}3}}^\infty \re^{-ik(x+\delta)-ik^3t}\frac{\rho(-k)}{\Delta(-k)} \D k, \\
        J_4 &= \int_{\abs\delta^{\frac{-1}3}}^\infty \re^{-ik^3t}\frac{\rho(-k)}{\Delta(-k)} \re^{-i\alpha^2kx}\left[1-\re^{-ik\alpha^2\delta}\right] \D k.
    \end{align*}
    We shall argue that each has limit $0$ as $\delta\to0$.
    The argument for the integral along the other ray of $\partial E_2^+$ is analogous, so shall be omitted.
    
    Let
    \[
        m = \sup_{x\in[0,\infty)} \abs{ \frac{\rho(-k)}{\Delta(-k)} },
    \]
    which, from our earlier analysis of the asymptotic behaviour of these functions and the locus of the zeros of $\Delta$, is finite.
    Then
    \begin{align*}
        \abs{J_1} &\leq m \left( \int_0^{\abs\delta^{\frac{-1}3}} \abs{1-\re^{-ik\delta}} \D k + \int_0^{\abs\delta^{\frac{-1}3}} \abs{1-\re^{-ik\alpha^2\delta}} \D k \right) \\
        &\leq m \left( \int_0^{\abs\delta^{\frac{-1}3}} k\abs{\delta} \D k + \int_0^{\abs\delta^{\frac{-1}3}} 2k\abs{\delta} \D k \right),
    \end{align*}
    where the inequality in the first integrand is bounding the chord by the arclength, and the second follows from the fact that
    \[
        \abs{1-\re^{-i \alpha^2k}}^2 = \left( 1-\re^{-i \alpha^2k} \right) \left( 1-\re^{i \alpha k} \right) = k^2-k^33\sqrt3 + \bigoh{k^4}
    \]
    as $\abs k \to 0$.
    Hence $\abs{J_1}\to0$ as $\delta\to0$.
    We also bound
    \[
        \abs{J_4} \leq 2m \int_{\abs\delta^{\frac{-1}3}}^\infty \re^{-kx\sqrt3/2} \D k,
    \]
    which has limit $0$ as $\delta\to0$.
    
    Integrating by parts twice in the numerator, we find that
    \begin{equation} \label{eqn:rho.on.Delta}
        \frac{\rho(-k)}{\Delta(-k)} = \frac{i\alpha}{k}\sum_{j=1}^{n-1} \re^{iky_j}\delta_j - \frac{\sqrt3u_0(1)}k \re^{iky_n} - \frac{i\alpha u_0(0)}k\re^{iky_0} + \bigoh{k^{-2}},
    \end{equation}
    in which $[0=y_0,y_1,\ldots,y_n=1]$ is a partition such that each $u_0\rvert_{(y_j,y_{j+1})}$ is extensible to a $\ContinuousSpace^2[y_j,y_{j+1}]$ function, and $\delta_j=u_0(y_j^+)-u_0(y_j^-)$.
    We split each of $J_2$ and $J_3$ into $n+2$ terms: one for each of the oscillatory behaviours $\re^{iky_j}$ on the right of equation~\eqref{eqn:rho.on.Delta}, and one for the asymptotic term.
    The asymptotic term yields an absolutely convergent improper integral for each of $J_2$ and $J_3$; in the case of $J_3$, the convergence is uniform in the $\delta$ appearing in the integrand.
    Hence, in the limit $\delta\to0$, we have convergence to $0$ of the part of each of $J_2$ and $J_3$ resulting from the $\bigoh{k^{-2}}$ term.
    
    Let $\phi_\delta(k) = k^3t+k(x+\delta)$.
    Then there exists an $M>0$ such that, for sufficiently large $a>0$, all $b>a$, and all sufficiently small $\delta>0$, we have that $\phi_\delta''(k)>0$ and $\phi_\delta'(k)-y_j>M>0$.
    Therefore,
    \begin{align*}
        \abs{\int_a^b \re^{i(y_jk-\phi(k))} \D k}
        &=
        \abs{ \frac{\re^{i(y_jb-\phi_\delta(b))}}{i(y_j-\phi_\delta'(b))} - \frac{\re^{i(y_ja-\phi_\delta(a))}}{i(y_j-\phi_\delta'(a))} - \int_a^b \frac{\phi''(k) \re^{i(y_jk-\phi_\delta(k))}}{i(y_j-\phi_\delta'(k))^2} \D k } \\
        &\leq
        \frac1{\phi_\delta'(b)-y_j} + \frac1{\phi_\delta'(a)-y_j} + \int_a^b \frac{\phi_\delta''(k)}{(\phi_\delta'(k)-y_j)^2} \D k \\
        &= \frac2{\phi_\delta'(a)-y_j} < \frac2M.
    \end{align*}
    Hence, by Dirichlet's test for convergence of improper integrals, each of the other $n+1$ constituents of each of $J_2$ and $J_3$ converges.
    Taking the limit $\delta\to0$,
    these $n+1$ constituents of each of $J_2$ and $J_3$ converge to $0$.
    So also do the whole of $J_2$ and $J_3$ converge to $0$ in the $\delta\to0$ limit.
\end{proof}

\newpage
\bibliographystyle{amsplain}
\bibliography{references}

\providecommand{\bysame}{\leavevmode\hbox to3em{\hrulefill}\thinspace}
\providecommand{\MR}{\relax\ifhmode\unskip\space\fi MR }
\providecommand{\MRhref}[2]{%
  \href{http://www.ams.org/mathscinet-getitem?mr=#1}{#2}
}
\providecommand{\href}[2]{#2}
\begin{thebibliography}{10}

\bibitem{ABS2022a}
S.~A. Aitzhan, S.~Bhandari, and D.~A. Smith, \emph{Fokas diagonalization of
  piecewise constant coefficient linear differential operators on finite
  intervals and networks}, Acta Appl. Math. \textbf{177} (2022), no.~2, 1--66.

\bibitem{bsz}
Jerry~L Bona, Shu~Ming Sun, and Bing-Yu Zhang, \emph{A nonhomogeneous
  boundary-value problem for the {K}orteweg--de {V}ries equation posed on a
  finite domain}, Communications in Partial Differential Equations \textbf{28}
  (2003), no.~7--8, 1391--1436.

\bibitem{BFPS}
L.~Boulton, G~Farmakis, B.~Pelloni, and D.A. Smith, \emph{Jumps and cusps:
  Talbot effect in non-periodic dispersive pdes}, arXiv
  http://arxiv.org/abs/2403.01117, 2024.

\bibitem{BFP}
Lyonell Boulton, George Farmakis, and Beatrice Pelloni, \emph{Beyond periodic
  revivals for linear dispersive {PDE}s}, Proceedings of the Royal Society A
  \textbf{477} (2021), no.~2251, 20210241.

\bibitem{boulton2020new}
Lyonell Boulton, Peter~J. Olver, Beatrice Pelloni, and David~A. Smith,
  \emph{New revival phenomena for linear integro--differential equations},
  Studies in Applied Mathematics \textbf{147} (2021), no.~4, 1209--1239.

\bibitem{deconinck2014method}
Bernard Deconinck, Thomas Trogdon, and Vishal Vasan, \emph{The method of fokas
  for solving linear partial differential equations}, SIAM Review \textbf{56}
  (2014), no.~1, 159--186.

\bibitem{DunfSchv2}
N~Dunford and J~T Schwartz, \emph{Linear operators, part {II} spectral theory
  self adjoint operators in hilbert spaces}, Wiley Interscience, 1963.

\bibitem{erdougan2016dispersive}
M~Burak Erdo{\u{g}}an and Nikolaos Tzirakis, \emph{Dispersive partial
  differential equations. wellposedness and applications}, vol.~86, London
  Mathematical Society Student Texts. Cambridge University Press, Cambridge,
  2016.

\bibitem{farmakis2024}
G~Farmakis, \emph{Revivals in time-evolution quasi-periodic problems}, arXiv
  preprint, 2311.02780, 2023.

\bibitem{farmakis2023new}
George Farmakis, Jing Kang, Peter~J Olver, Changzheng Qu, and Zihan Yin,
  \emph{New revival phenomena for bidirectional dispersive hyperbolic
  equations}, arXiv preprint arXiv:2309.14890.

\bibitem{fokas2005transform}
AS~Fokas and Beatrice Pelloni, \emph{A transform method for linear evolution
  pdes on a finite interval}, IMA journal of applied mathematics \textbf{70}
  (2005), no.~4, 564--587.

\bibitem{smithfokas}
AS~Fokas and David~A Smith, \emph{Evolution pdes and augmented eigenfunctions.
  finite interval}, Advances in Differential Equations \textbf{21} (2016),
  no.~7/8, 735--766.

\bibitem{Gasdia19}
Forrest Gasdia, \emph{{R}oots{A}nd{P}oles.jl: Global complex roots and poles
  finding in the {J}ulia programming language}, 2019.

\bibitem{jackson}
Dunham Jackson, \emph{Expansion problems with irregular boundary conditions},
  Proceedings of the American Academy of Arts and Sciences \textbf{51}.

\bibitem{Kow2015a}
Piotr Kowalczyk, \emph{Complex root finding algorithm based on delaunay
  triangulation}, ACM Trans. Math. Softw. \textbf{41} (2015), no.~3.

\bibitem{Kow2018a}
\bysame, \emph{Global complex roots and poles finding algorithm based on phase
  analysis for propagation and radiation problems}, IEEE Transactions on
  Antennas and Propagation \textbf{66} (2018), no.~12, 7198--7205.

\bibitem{langer}
Rudolph~E Langer, \emph{On the zeros of exponential sums and integrals}, Bull.
  AMer. Math. Soc. \textbf{37} (1931).

\bibitem{mclau}
Kenneth DT-R McLaughlin and Nigel~JE Pitt, \emph{On ringing effects near jump
  discontinuities for periodic solutions to dispersive partial differential
  equations}, arXiv preprint arXiv:1107.1571 (2011).

\bibitem{olver2010dispersive}
Peter~J Olver, \emph{Dispersive quantization}, The American Mathematical
  Monthly \textbf{117} (2010), no.~7, 599--610.

\bibitem{olver2018revivals}
Peter~J Olver, Natalie~E Sheils, and David~A Smith, \emph{Revivals and
  fractalisation in the linear free space {S}chr\"{o}dinger equation},
  Quarterly of Applied Mathematics \textbf{78} (2020), no.~2, 161--192.

\bibitem{ApproxFun.jl-2014}
Sheehan Olver and Alex Townsend, \emph{A practical framework for
  infinite-dimensional linear algebra}, Proceedings of the 1st Workshop for
  High Performance Technical Computing in Dynamic Languages -- HPTCDL `14,
  {IEEE}, 2014.

\bibitem{pelloni2005spectral}
Beatrice Pelloni, \emph{The spectral representation of two-point boundary-value
  problems for third-order linear evolution partial differential equations},
  Proceedings of the Royal Society A: Mathematical, Physical and Engineering
  Sciences \textbf{461} (2005), no.~2061, 2965--2984.

\bibitem{pelloni2013spectral}
Beatrice Pelloni and David~A Smith, \emph{Spectral theory of some
  non-selfadjoint linear differential operators}, Proceedings of the Royal
  Society A: Mathematical, Physical and Engineering Sciences \textbf{469}
  (2013), no.~2154, 20130019.

\bibitem{smith2012}
David~A. Smith, \emph{Well-posed two-point initial-boundary value problems with
  arbitrary boundary conditions},  \textbf{152} (2012), 473--496.

\bibitem{Smi2012b}
\bysame, \emph{Well-posedness and conditioning of 3rd and higher order
  two-point initial-boundary value problems}, arXiv:1212.5466, 2012.

\bibitem{smith2020revival}
\bysame, \emph{Revivals and fractalization}, Dynamical System Web (2020), 1--8.

\end{thebibliography}

\end{document}